\documentclass[11pt,reqno]{amsart}

\usepackage{amssymb,amsmath,epsfig}
\usepackage{amsfonts}
\usepackage{pst-plot}

\newcommand{\Z}{\mathbb{Z}}
\newcommand{\R}{\mathbb{R}}
\newcommand{\C}{\mathbb{C}}

\newcommand{\Q}{\mathbb{Q}}
\newcommand{\SL}{{\text {\rm SL}}}
\newcommand{\im}{\textnormal{Im}}
\newcommand{\sgn}{\operatorname{sgn}}
\def\H{\mathbb{H}}
\newcommand{\leg}[2]{\genfrac{(}{)}{}{}{#1}{#2}}
\newtheorem{theorem}{Theorem}[section]

\newtheorem{lemma}[theorem]{Lemma}
\newtheorem{corollary}[theorem]{Corollary}
\newtheorem{proposition}[theorem]{Proposition}

\newtheorem*{theorem*}{Theorem}
\newtheorem{remark}[theorem]{Remark}

\numberwithin{equation}{section}

\title[Generalized odd rank moments and odd Durfee symbols]
{Automorphic properties of generating functions for generalized odd rank moments and odd Durfee symbols}

\date{\today}
\author{Claudia Alfes} 
\author{Kathrin Bringmann} 
\author{Jeremy Lovejoy}
\address{Fachbereich Mathematik, Technische Universit\"at Darmstadt, Schlossgartenstrasse 7,
64289 Darmstadt, Germany}

\address{Mathematical Institute, University of Cologne, Weyertal 86-90, 50931 Cologne, Germany}

\address{CNRS, LIAFA, Universit\'e Denis Diderot,
Case 7014, 75205 Paris Cedex 13, FRANCE}

\email{alfes@mathematik.tu-darmstadt.de}

\email{kbringma@math.uni-koeln.de}

\email{lovejoy@liafa.jussieu.fr}

\thanks{The second author was partially supported by NSF grant DMS-0757907 and by the Alfried Krupp prize.}

\begin{document}
\begin{abstract}
We define two-parameter generalizations of Andrews' $(k+1)$-marked odd Durfee symbols and $2k$th symmetrized odd rank moments, and study the automorphic properties of some of their generating functions.  When $k=0$ we obtain families of modular forms and mock modular forms.  When $k \geq 1$, we find quasimodular forms and quasimock modular forms.
\end{abstract}

\maketitle

\section{Introduction}
An effective method for discovering $q$-series with interesting number-theoretic behavior is to generalize the combinatorics of partitions.  Perhaps the most striking example of this is work of the second two authors relating the rank of an overpartition to the Hurwitz class numbers \cite{Br-Lo2}.  Another example is work of Osburn and the second two authors, where extensions of Andrews' $(k+1)$-marked Durfee symbols and $2k$th symmetrized rank moments to overpartition pairs led to many new quasimock modular forms \cite{Br-Lo-Os1}.

Quasimock modular forms combine the properties of classical quasimodular forms and mock modular forms, which themselves generalize Ramanujan's mock theta functions.  Ramanujan's mock theta functions are $q$-hypergeometric series like
\begin{equation*}
f(q):= \sum_{n\ge 0} \frac{q^{n^2}}{(1+q)^2\cdots (1+q^n)^2}
\end{equation*}
whose behavior is closely related to that of modular forms.  To be more precise, Zwegers \cite{Zw} has ``completed" the mock theta functions to obtain so-called harmonic weak Maass forms, which are certain non-holomorphic modular forms (see Section \ref{MockSection} for the definition).
For this he required additional (classical) modular forms which are related to each of the mock theta functions and which we call, following Zagier, the \textit{shadow} of the mock theta function (again see Section \ref{MockSection} for the precise definition).  All mock theta functions turn out to be holomorphic parts of harmonic Maass forms, and their shadows are all unary weight $\frac32$ theta functions.
A \textit{mock modular form} is then more generally the holomorphic part of any harmonic weak Maass form of weight $k$, and the associated shadow is then a modular form of weight $2-k$.
Recall that a \textit{quasi modular form} may be defined as a meromorphic functions on the upper half-plane
that can be written as a linear combination of derivatives of modular forms.
In analogy, a \textit{quasimock modular form} is a linear combination of derivatives of mock modular forms. Some of the applications of these constructions will be mentioned in Section \ref{ConclusionSection}.

In the present paper we consider generalizations of Andrews' $(k+1)$-marked \emph{odd} Durfee symbols and the $2k$th symmetrized \emph{odd} rank moments.  In Section 2 we describe these generalized combinatorial objects and derive their generating functions, which turn out to be the series
\begin{equation} \label{fund}
N^o(a,b;z;q) :=
\sum_{n \geq 0} \frac{\left(-q/a,-q/b;q^2\right)_n(ab)^nq^{2n+1}}
{\left(zq,q/z;q^2\right)_{n+1}}
\end{equation}
or some of its derivatives,
\begin{equation} \label{fund2}
\mathcal{N}_{2k}^o(a,b;q) := \frac{1}{(2k)!}\left(\frac{d^{2k}}{dz^{2k}}z^kN^o(a,b;z;q)\right) \bigg |_{z=1}.
\end{equation}
Here we have employed the standard basic hypergeometric series notation,
$$
\left(a_1,a_2,\dots,a_j;q \right)_n :=
\frac{\left(a_1,a_2,\dots,a_j;q\right)_{\infty}}{\left(a_1q^n,a_2q^n,\dots,a_jq^n;q\right)_{\infty}},
$$
where
$$
(a_1,a_2,\dots,a_j;q)_{\infty} :=
\prod_{k=0}^{\infty}\left(1-a_1q^k\right)\left(1-a_2q^k\right)\cdots\left(1-a_jq^k\right),
$$
and as is typical we drop the ``$;q$" unless the base is something other than $q$.

Then we study the automorphic properties of some special cases of the generating functions $N^o(a,b;z;q)$ and $\mathcal{N}_{2k}^o(a,b;q)$, beginning in Section 3 with the case $b=1/a$, where $q$-series identities can be used to show that the function $N^o(a,1/a;z;q)$ is essentially an infinite product.
\begin{theorem} \label{intro.5}
We have
\begin{equation} \label{intro.5eq}
 N^o(a,1/a;z;q)
 +  \frac{1}{(z+a)\left(1+1/az\right)}
 = \frac{\left(-aq,-q/a;q^2\right)_{\infty}}{\left(z+a\right)\left(1+1/az\right)\left(zq,q/z;q^2\right)_{\infty}}.
\end{equation}
\end{theorem}

\noindent Standard facts about Jacobi forms then imply the following two corollaries.

\begin{corollary} \label{introcor1}
If $z$ and $a$ are of the form $\zeta q^c$ for $c \in \mathbb{Q}$, $\zeta$ a root of unity, and $z \not \in \{-1/a,-a\}$, then $$\frac{1}{(z+a)(1+1/az)}+ N^o\left(a,1/a;z;q\right)$$ is a modular form.
\end{corollary}

\begin{corollary} \label{introcor2}
If $a = \zeta q^c \neq -1$ for $c \in \mathbb{Q}$ and $\zeta$ a root of unity, then $\mathcal{N}_{2k}(a,1/a;q)$ is a quasimodular form.
\end{corollary}


\begin{remark}
\emph{We  point out that the assertions about the modularity of functions in this paper are in general ``up to multiplication by a power of $q$" and this will only be made precise for the mock modular forms in Theorem \ref{intro1} (see Theorem \ref{explicit}).  Moreover, a substitution of the form $q\mapsto q^M$ is typically required so that the modularity is with respect to some congruence subgroup $\Gamma_1(N)$.  We shall not determine any of these subgroups explicitly.}
\end{remark}

In Section $4$ we look at mock modular forms arising from $N^o(a,b;z;q)$, of which
there are already a number of important examples.  For instance,
$q^{-1}N^o(0,0;1;q)$ is Watson's third order mock theta function $\omega(q)$ \cite{Wa1},
\begin{equation*}
\omega(q) = q^{-1}N^o(0,0;1;q) = \sum_{n \geq 0} \frac{q^{2n^2+2n}}{(q;q^2)_{n+1}^2},
\end{equation*}
$N^o(0,1;1;q)$ and $q^{-1}N^o(0,1/q;1;q)$ are McIntosh's second order mock theta functions
$A(q)$ and $B(q)$ \cite{Mc1} (which were also studied in \cite{Br-On-Rh1}),
\begin{equation*}
A(q) = N^o(0,1;1;q) = \sum_{n \geq 0} \frac{q^{(n+1)^2}(-q;q^2)_n}{(q;q^2)_{n+1}^2},
\end{equation*}
\begin{equation*}
B(q) = q^{-1}N^o(0,1/q;1;q) = \sum_{n \geq 0} \frac{q^{n^2+n}(-q^2;q^2)_n}{(q;q^2)_{n+1}^2},
\end{equation*}
$q^{-1}N^o(1,1/q;1;q)$ is the Hikami-Ramanujan mock theta function $h_1(q)$ \cite[Eq. (12), corrected]{Hi1},\cite[p.3, $\phi(q)$]{Ra1},
\begin{equation*}
h_1(q) = q^{-1}N^o(1,1/q;1;q) = \sum_{n \geq 0} \frac{q^n(-q)_{2n}}{(q;q^2)_{n+1}^2},
\end{equation*}
$N^o(0,1;i;q)$ is the eighth order mock theta function $U_1(q)$ of Gordon and McIntosh \cite{Go-Mc1,Mc1},
\begin{equation*}
U_1(q) = N^o(0,1;i;q) = \sum_{n \geq 0} \frac{q^{(n+1)^2}(-q;q^2)_n}{(-q^2;q^4)_{n+1}},
\end{equation*}
and $N^o(1,1/q;i;q)$ is the mock theta function $\lambda(q)$ studied by both Andrews \cite{An.5} and McIntosh \cite{Mc1}
(which is also equal to the eighth order mock theta function $V_1(q)$ \cite{Go-Mc1}),
\begin{equation*}
\lambda(q) = N^o(1,1/q;i;q) = \sum_{n \geq 0} \frac{q^{n+1}(-q)_{2n}}{(-q^2;q^4)_{n+1}}.
\end{equation*}


Combining $q$-series identities with work of Zwegers \cite{Zw1}, we shall see that there are many more mock modular forms among the functions $N^o(a,b;z;q)$ than just the ones above.
\begin{theorem} \label{intro1}
Let $(a,b) \in \{(0,0), (0,1/q), (0,-1), (1,-1)$, $(1,1/q)$\} and let $z$ be any root of unity.  If $(a,b,z) \not \in \{(1,-1,\pm 1), (0,-1,1), (1,1/q,-1)\}$, then the series $N^o(a,b;z;q)$ is a mock theta function.  If $(a,b,z) \in \{(1,-1,\pm 1), (0,-1,1), (1,1/q,-1)\}$, then the series $N^o(a,b;z;q)$ is a weight $3/2$ mock modular form.
\end{theorem}

In Section $5$ we take up the exceptional triples from Theorem \ref{intro1} and show that in each case there is a connection with class numbers of binary quadratic forms.  Let $H(n)$ denote the Hurwitz class number and $F(n)$ denote the Kronecker class number.
\begin{theorem} \label{intro1.5}
We have
\begin{eqnarray}
N^o(1,-1;1;q) &=& -N^o(1,-1;-1;-q) = \sum_{n \geq 1} 2F(n)q^n, \label{classeq1} \\
N^o(0,-1;1;q) &=& \sum_{n \geq 1} H(8n-1)q^n, \label{classeq2} \\
N^o(1,1/q;-1;q) &=& -\sum_{n \geq 1} F(4n-1)(-q)^n = 3\sum_{n \geq 1} H(8n-5)q^{2n-1} - \sum_{n \geq 1} H(8n-1)q^{2n}. \label{classeq3}
\end{eqnarray}
\end{theorem}
\noindent These follow from $q$-series identities together with work of Andrews \cite{An.7}, Humbert \cite{Hu1}, Kronecker \cite{Kr1}, and Watson \cite{Wa2}.

Finally in Section $6$ we prove the following Theorem, which depends on certain partial differential equations involving $N^o(a,b;z;q)$.
\begin{theorem} \label{intro2}
For $k \geq 1$ and $(a,b) = (0,0), (0,1/q), (0,-1), (1,-1)$, or $(1,1/q)$, the series $\mathcal{N}_{2k}^{o}(a,b;q)$ is a quasimock modular form.
\end{theorem}

\section{Generalized odd Durfee symbols and odd rank moments}

The notation here and throughout is intended to be reminiscent of that of \cite{An1} and \cite{Br-Lo-Os1}.  By a \emph{generalized odd Durfee symbol} for the positive integer $n$ we mean a two-rowed array with a triple subscript,
\begin{equation*} 
\begin{pmatrix}
a_1 & a_2 & \cdots & a_i \\
b_1 & b_2 & \cdots & b_j
\end{pmatrix}_{\lambda,\mu,t},
\end{equation*}
where $t \geq 0$, $\lambda = (\lambda_1, \lambda_2,\dots \lambda_h)$ and $\mu = (\mu_1,\mu_2,\dots,\mu_k)$ are partitions into distinct odd parts of size at most $2t-1$, each row is a partition into odd parts of size at most $2t+1$, and
$$
n = \left(a_1+a_2 \cdots + a_i\right) + \left(b_1+ b_2 + \cdots b_j \right) +  \left(\lambda_1+ \lambda_2 + \cdots \lambda_h \right) + \left(\mu_1+ \mu_2 + \cdots \mu_k \right) + 2t+1.
$$
For example, the two-rowed array
\begin{equation*} 
\begin{pmatrix}
3 & 1 & 1  \\
9 & 7 & 3 & 3 & 3
\end{pmatrix}_{(7,3),(5,3,1),4}
\end{equation*}
is a generalized odd Durfee symbol for $58$.

We call these generalized odd Durfee symbols because when $\lambda$ and $\mu$ each contain all of the odd numbers between $1$ and $2t-1$, then we have one of the ordinary odd Durfee symbols of Andrews \cite{An1}.  It is natural then to define an odd number $2x-1$ as \emph{missing} from a partition $\nu$ into distinct odd parts of size at most $2t-1$ if $1 \leq x \leq t$ and $2x-1$ doesn't occur in $\nu$.  For instance, in the example above $t=4$ and so $\lambda = (7,3)$ has two missing numbers and $\mu = (5,3,1)$ has one missing number.  As Andrews did in the case of ordinary odd Durfee symbols, we define the \emph{rank} of an odd Durfee symbol to be the number of entries on the top row minus the number of entries on the bottom row of the generalized odd Durfee symbol.

It is now straightforward, using the elementary theory of partitions, to see that if $N^o(r,s,m,n)$ denotes the number of generalized odd Durfee symbols for $n$, where $r$ is the number of missing parts in $\lambda$, $s$ is the number of missing parts in $\mu$, and $m$
is the rank, then
\begin{equation*}
N^o(a,b;z;q) = \sum_{r,s,n \geq 0 \atop m \in \mathbb{Z}} N^o(r,s,m,n)a^rb^sz^mq^n.
\end{equation*}
Now the reader should have no trouble interpreting any given instance of $N^o(a,b;z;q)$ combinatorially.  To give an example, $N^o(1,-1;1;q)$ is the generating function for generalized odd Durfee symbols in which $\lambda = \mu$, each symbol being counted with the weight $(-1)^r$.


To get at the functions $\mathcal{N}_{2k}^o(a,b;q)$, we consider the rank moments of generalized odd Durfee symbols.  The \emph{$k$th symmetrized rank moment} $\eta_k^o(r,s,n)$ is defined by
\begin{equation*}
\eta_k^o(r,s,n) := \sum_{m \in \mathbb{Z}} \begin{pmatrix} m + \lfloor \frac{k}{2} \rfloor \\ k \end{pmatrix}N^o(r,s,m,n).
\end{equation*}
In light of the invariance $z \leftrightarrow 1/z$ in \eqref{fund}, we have
\begin{equation} \label{invariance}
N^o(r,s,m,n) = N^o(r,s,-m,n),
\end{equation}
and hence $\eta_k^o(r,s,n) = 0$ whenever $k$ is odd.  As for $k$ even, we have the following:

\begin{theorem} \label{gfsymmoment}
For $k \geq 1$,
\begin{eqnarray*}
\sum_{r,s,n \geq 0} \eta_{2k}^o(r,s,n)a^r b^sq^n &=& \mathcal{N}_{2k}^o(a,b;q) \\
&=& \frac{\left(-aq,-bq;q^2\right)_{\infty}}{\left(q^2,abq^2;q^2\right)_{\infty}} \sum_{n \in \mathbb{Z}} \frac{\left(-q/a,-q/b;q^2\right)_n(-ab)^nq^{n^2+3n+1 + k(2n+1)}}{\left(1-q^{2n+1}\right)^{2k+1}\left(-aq,-bq;q^2\right)_{n+1}}.
\end{eqnarray*}
\end{theorem}

\begin{proof}
The first equality is straightforward from the definition \eqref{fund2}.  For the second, we begin with the identity
\begin{equation}
\begin{split}
 \label{Watsonid}
N^o(a,b;z;q) &= \frac{\left(-aq,-bq;q^2\right)_{\infty}}{2\left(q^2,abq^2;q^2\right)_{\infty}} \sum_{n \in \mathbb{Z}} \frac{\left(1-q^{4n+2}\right)\left(-q/a,-q/b;q^2\right)_n(-ab)^nq^{n^2+3n+1}}{\left(1-zq^{2n+1}\right)\left(1-q^{2n+1}/z\right)
\left(-aq,-bq;q^2\right)_{n+1}} \\
&=
\frac{\left(-aq,-bq;q^2\right)_{\infty}}{\left(q^2,abq^2;q^2\right)_{\infty}} \sum_{n \in \mathbb{Z}} \frac{\left(-q/a,-q/b;q^2\right)_n(-ab)^nq^{n^2+3n+1}}{\left(1-zq^{2n+1}\right)\left(-aq,-bq;q^2\right)_{n+1}}.
\end{split}
\end{equation}
The first equation follows from the case $(a,b,c,d,e,q) \rightarrow (q^2,zq,q/z,-q/a,-q/b,q^2)$ of a limiting case of the Watson-Whipple transformation \cite[p.242, Eq. (III.18), $n \to \infty$]{Ga-Ra1},
\begin{equation*} \label{Watson-Whipple}
\sum_{n=0}^{\infty}
\frac{\left(aq/bc,d,e\right)_n(\frac{aq}{de})^n}{\left(q,aq/b,aq/c\right)_n} =
\frac{\left(aq/d,aq/e\right)_{\infty}}{\left(aq,aq/de\right)_{\infty}}
\sum_{n=0}^{\infty}
\frac{\left(a,\sqrt{a}q,-\sqrt{a}q,b,c,d,e\right)_n(aq)^{2n}(-1)^nq^{n(n-1)/2}}
{\left(q,\sqrt{a},-\sqrt{a},aq/b,aq/c,aq/d,aq/e\right)_n(bcde)^n}.
\end{equation*}
Substituting $n \mapsto -n-1$ and simplifying using the identity
\begin{equation} \label{negative}
(a)_{-n} = \frac{(-1)^nq^{\frac{n(n+1)}{2} }}{a^n\left(a^{-1}q \right)_n}
\end{equation}
converts the unilateral sum to a bilateral sum.  The second equation follows from after applying the partial fraction identity
$$
\frac{1-q^{4n+2}}{(1-zq^{2n+1})(1-q^{2n+1}/z)} = \frac{1}{1-zq^{2n+1}} + \frac{z^{-1}q^{2n+1}}{1-q^{2n+1}/z},
$$
and again using the substitution $n \mapsto -n-1$ to simplify one of the two resulting series.

Next, we differentiate (as in \cite[p.63]{An1}) to obtain
$$
\begin{aligned}
\sum_{r,s,n \geq 0} & \eta_{2k}^o(r,s,n)q^n  
= \mathcal{N}_{2k}^o(a,b;q) 
= \frac{1}{(2k)!} \frac{d^{2k}}{dz^{2k}}\left(z^kN^o(a,b;z;q)\right) \bigg |_{z=1}
\\ &
= \frac{1}{(2k)!} \sum_{j=0}^{2k} \binom{2k}{j} k(k-1)\cdots(k-j+1)
 \frac{d^{2k-j}}{dz^{2k-j}}\left(N^{o}(a,b;z;q) \right) \bigg |_{z=1} \\
&= \sum_{j=0}^k \binom{k}{j} \frac{\left(-aq,-bq;q^2\right)_{\infty}}{\left(q^2,abq^2;q^2\right)_{\infty}} \sum_{n \in \mathbb{Z}} \frac{\left(-q/a,-q/b;q^2\right)_n(-ab)^nq^{n^2+3n+1+(2k-j)(2n+1)}}{\left(1-q^{2n+1}\right)^{2k-j+1}\left(-aq,-bq;q^2\right)_{n+1}} \\
& = \frac{\left(-aq,-bq;q^2\right)_{\infty}}{\left(q^2,abq^2;q^2\right)_{\infty}} \sum_{n \in \mathbb{Z}} \frac{\left(-q/a,-q/b;q^2\right)_n(-ab)^nq^{n^2+3n+1+k(2n+1)}}{\left(1-q^{2n+1}\right)^{2k+1}\left(-aq,-bq;q^2\right)_{n+1}},
\end{aligned}
$$
as desired.
\end{proof}

\begin{remark} \label{bigremark}
\emph{
In addition to the symmetrized rank moment, it is also useful to consider the ordinary rank moment $H_k^o(r,s,n)$, defined by
\begin{equation*}
H_k^o(r,s,n) := \sum_{m \in \mathbb{Z}} m^k N^o(r,s,m,n).
\end{equation*}
Let $\mathcal{H}_k^o(a,b;q)$ denote the three-variable generating function for $H_k^o(r,s,n)$, i.e.,
\begin{equation*}
\mathcal{H}_k^o(a,b;q) := \sum_{r,s,n \geq 0} H_k^o(r,s,n) a^rb^sq^n.
\end{equation*}
While $\mathcal{H}_{2k}^o(a,b;q)$ doesn't have a generating function as elegant as the one for $\mathcal{N}_{2k}^o(a,b;q)$ in Theorem \ref{gfsymmoment}, it does satisfy
\begin{equation*}
\mathcal{H}_{2k}^o(a,b;q) = \delta_z^{2k}\left(  N^o(a,b;z;q)\right) \bigg |_{z=1},
\end{equation*}
where
$\delta_z := z \frac{d}{dz}$,
and so it fits in more naturally with the theory of Jacobi forms.  Moreover, using the fact that $\mathcal{H}_{2k-1}^o(a,b;q) = 0$ (which follows from \eqref{invariance}) we have that any $\mathcal{N}_{2k}^o(a,b;q)$ may be written as a linear combination of the $\mathcal{H}_{2k}^o(a,b;q)$ (and vice versa).  Hence any automorphic properties are shared by these two generating functions.
}
\end{remark}

By a \emph{$k$-marked generalized odd Durfee symbol} for $n$ we mean a generalized odd Durfee symbol for $n$ where the entries in the array may now occur in $k$ colors (denoted by subscripts $1,\dots ,k$), such that
\begin{enumerate}
\item
The sequence of parts and the sequence of subscripts in each row is non-increasing.
 \item Each   of the subscripts $1,2,...,k-1$ occurs at least once in the top row.
  \item If  $M_1, N_2, \dots V_{k-2},W_{k-1}$ are the largest parts with their respective subscripts in the top row, then all parts in the bottom row with subscript $1$ lie in the interval $[1,M_1]$, with subscript $2$ lie in $[M_1,N_2]$, $\dots$, with subscript $k-1$ lie in $[V_{k-2},W_{k-1}]$, and with subscript $k$ lie in $[W_{k-1},t]$, where $t$ is the third subscript of the symbol.
\end{enumerate}
If  the subscripts $\lambda$ and $\mu$ have no missing numbers, then this   is precisely the definition of Andrews' $k$-marked odd Durfee symbols.

Let $\mathcal{D}_k^o(r,s,n)$ be the number of generalized $k$-marked odd
Durfee symbols having $r$ missing parts in the subscript $\lambda$ and $s$ missing parts in the subscript $\mu$.  For such a
symbol $\delta$ and for each $i$ we denote the number of entries
in the top (resp. bottom) row with subscript $i$ by
$\tau_i(\delta)$ (resp. $\beta_i(\delta)$).
We extend the definition of rank by defining the \textit{$i$th rank} of a generalized $k$-marked odd Durfee symbol $\delta$ to be
\begin{equation*} \label{ithrank}
\rho_i(\delta) :=
\begin{cases}
\tau_i(\delta) - \beta_i(\delta) - 1 & \text{for $1 \leq i <
k$},\\
\tau_i(\delta) - \beta_i(\delta)& \text{for $i = k$.}
\end{cases}
\end{equation*}

Let $\mathcal{D}_k^o(r,s,m_1,m_2,\dots,m_k,n)$ denote the number of
generalized $k$-marked odd Durfee symbols counted by $\mathcal{D}_k^o(r,s,n)$ with $i$th rank equal to $m_i$.
We have the following generating function:
\begin{theorem} \label{genDsymbolgf}
For $k \geq 2$ we have
\begin{equation} \label{genDsymbolgfeq}
\begin{aligned}
&\sum_{m_1,m_2,\dots,m_k \in \mathbb{Z}} \sum_{r,s,n \geq 0}
\mathcal{D}_k^o(r,s,m_1,m_2,\dots,m_k,n) x_1^{m_1}x_2^{m_2}\cdots
x_k^{m_k} d^re^sq^n \\
&= \frac{(-aq,-bq;q^2)_{\infty}}{(q^2,abq^2;q^2)_{\infty}} \sum_{n \geq 0} \frac{(1-q^{4n+2})(-q/a,-q/b;q^2)_n(-ab)^nq^{n^2+(2k+1)n+k}}
{(-aq,-bq;q^2)_{n+1}\prod_{i=1}^k(1-x_iq^{2n+1})(1-q^{2n+1}/x_i)}.
\end{aligned}
\end{equation}
\end{theorem}

\begin{proof}
Following Andrews \cite{An1}, we begin by appealing to the $k$-fold generalization of Watson's $q$-analogue of Whipple's theorem \cite[p.43, Eq. (2.4)]{An1}.  In that identity we replace $k$ by $k+1$, let $N \to \infty$, replace $q$ by $q^2$, let $a=q^2$, $b_{k+1} = -q/a$, $c_{k+1} = -q/b$, and for each $1 \leq i \leq k$, set $b_i = x_iq$ and $c_i = q/x_i$.  After some simplification the result is
\begin{equation*}
\begin{aligned}
\sum_{m_1,m_2,\dots,m_k \geq 0}& (-q/a,-q/b;q^2)_{m_1+\cdots+m_k}(ab)^{m_1+\cdots+m_k}q^{2(m_1+\cdots+m_k)+1} \\
&\times \frac{q^{2m_1+1}}{(x_1q,q/x_1;q^2)_{m_1+1}} \times \frac{q^{2(m_1+m_2)+1}}{(x_2q^{2m_1+1},q^{2m_1+1}/x_2;q^2)_{m_2+1}} \times \cdots \\
&\times \frac{q^{2(m_1+\cdots+m_{k-1})+1}}{(x_{k-1}q^{2(m_1+\cdots+m_{k-2})+1},q^{2(m_1+\cdots+m_{k-2})+1}/x_{k-1};q^2)_{m_{k-1}+1}} \\
&\times \frac{1}{(x_{k}q^{2(m_1+\cdots+m_{k-1})+1},q^{2(m_1+\cdots+m_{k-1})+1}/x_{k};q^2)_{m_{k}+1}} \\
&= \frac{(-aq,-bq;q^2)_{\infty}}{(q^2,abq^2;q^2)_{\infty}} \sum_{n \geq 0} \frac{(1-q^{4n+2})(-q/a,-q/b;q^2)_n(-ab)^nq^{n^2+(2k+1)n+k}}
{(-aq,-bq;q^2)_{n+1}\prod_{i=1}^k(1-x_iq^{2n+1})(1-q^{2n+1}/x_i)}.
\end{aligned}
\end{equation*}
That the left hand side above is the generating function
$$
\sum_{m_1,m_2,\dots,m_k \in \mathbb{Z}} \sum_{r,s,n \geq 0}
\mathcal{D}_k^o(r,s,m_1,m_2,\dots,m_k,n) x_1^{m_1}x_2^{m_2}\cdots
x_k^{m_k} d^re^sq^n
$$
follows just as in \cite{An1}.  Indeed, the only difference between this multiple sum and the multiple sum in \cite[p. 64, Eq. (9.1)]{An1} is that
our $(-q/a,-q/b;q^2)_{m_1+\cdots+m_k}(ab)^{m_1+\cdots+m_k}$ is replaced by $q^{2(m_1+\cdots+m_k)^2}$ (corresponding to the subscripts $\lambda$ and $\mu$ of the symbol each having no missing parts).
\end{proof}

Setting $x_i=1$ in \eqref{genDsymbolgfeq} we may conclude:
\begin{corollary}
For $k \geq 1$ we have
$\eta_{2k}^o(r,s,n) = \mathcal{D}_{k+1}^o(r,s,n)$.
\end{corollary}

This ends our discussion of the combinatorics of the functions $N^o(a,b;z;q)$ and $\mathcal{N}_{2k}^o(a,b;q)$.  We are now ready to study their automorphic properties.





\section{Modular forms and quasimodular forms}
In this section we prove Theorem \ref{intro.5} and Corollaries \ref{introcor1} and \ref{introcor2}.

\begin{proof}[Proof of Theorem \ref{intro.5}]
We require a $_3 \phi _2$ transformation \cite[p.241, Eq. (III.9)]{Ga-Ra1},
\begin{equation}\label{3phi2}
\sum_{n \geq 0} \frac{(a,b,c)_n\left(de/abc\right)^n}{(d,e,q)_n} = \frac{\left(e/a,de/bc\right)_{\infty}}{\left(e,de/abc;q\right)_{\infty}}
\sum_{n \geq 0} \frac{\left(a,d/b,d/c\right)_n(e/a)^n}{\left(d,de/bc,q\right)_n},
\end{equation}
as well as the $q$-Gauss summation \cite[p.236, Eq. (II.8)]{Ga-Ra1},
\begin{equation} \label{qGauss}
\sum_{n \geq 0} \frac{(a,b)_n \left(c/ab\right)^n}{(c,q)_n} = \frac{\left(c/a,c/b\right)_{\infty}}{\left(c,c/ab\right)_{\infty}}.
\end{equation}
Beginning with an application of \eqref{3phi2} with $(a,b,c,d,e,q) = (q^2,-q/a,-aq,q^3/z,zq^3,q^2)$, we compute
$$
\begin{aligned}
1 + &(z+a)(1+1/az)N^o(1/a,a;z;q) \\ &= 1 + q(z+a)(1+1/az)\sum_{n \geq 0}
\frac{(-aq^2/z,-q^2/az;q^2)_n(zq)^n}{(q/z,q^2;q^2)_{n+1}} \\
&= 1 + (1+a/z)(1+1/az)\sum_{n \geq 1}
\frac{(-aq^2/z,-q^2/az;q^2)_{n-1}(zq)^n}{(q/z,q^2;q^2)_{n}} \\
&= 1 + \sum_{n \geq 1}
\frac{(-aq^2/z,-q^2/az;q^2)_{n}(zq)^n}{(q/z,q^2;q^2)_{n}} \\
&= \sum_{n \geq 0}
\frac{(-a/z,-1/az;q^2)_{n}(zq)^n}{(q/z,q^2;q^2)_{n}}
= \frac{(-aq,-q/a;q^2)_{\infty}}{(zq,q/z;q^2)_{\infty}},
\end{aligned}
$$
the last equality following from the case $(a,b,c,q) = (-a/z,-1/az,q/z,q^2)$ of \eqref{qGauss}.
\end{proof}

\begin{proof}[Proof of Corollaries  \ref{introcor1} and \ref{introcor2}]
The product on the right-hand side of \eqref{intro.5eq} is the quotient of two Jacobi forms (with modular variable $\tau$, where $q := e^{2\pi i \tau}$, and with different elliptic variables, say $u$ and $v$, where $z := e^{2 \pi i u}$ and $w := e^{2 \pi i v}$).   Specializing the elliptic variable of a Jacobi form at torsion points (i.e. points of the form $\mathbb{Q}\tau + \mathbb{Q})$ is known to give modular forms.  This yields Corollary \ref{introcor1}.  It is also a fact that if $F(u;\tau)$ is a Jacobi form then operating with
$\frac{\partial^{\ell}}{\partial u^{\ell}} | _{u=0}$
gives a quasimodular form for $\ell \geq 2$ (and a modular form for $\ell = 1$).
Appealing to Remark \ref{bigremark}, this then implies Corollary  \ref{introcor2}.
For more on Jacobi forms, the reader may consult \cite{Ei-Za1}.
\end{proof}
\section{Mock theta functions and mock modular forms} \label{MockSection}
The goal of this section is to prove Theorem \ref{explicit}, which is a more precise version of Theorem \ref{intro1}.  First we present some background, beginning with definitions of harmonic weak Maass forms (of half-integral weight), mock theta functions, and mock modular forms \cite{BF,Za1,Zw1}.

If $k\in \frac{1}{2}\Z\setminus \Z$, $\tau=x+iy$ with $x, y\in \R$,
then the \textit{weight $k$ hyperbolic Laplacian} is given by
\begin{equation*}\label{laplacian}
\Delta_k := -y^2\left( \frac{\partial^2}{\partial x^2} +
\frac{\partial^2}{\partial y^2}\right) + iky\left(
\frac{\partial}{\partial x}+i \frac{\partial}{\partial y}\right).
\end{equation*}
If $v$ is odd, then define $\epsilon_v$ by
\begin{equation*}
\epsilon_v:=\begin{cases} 1 \ \ \ \ &{\text {\rm if}}\ v\equiv
1\pmod 4,\\
i \ \ \ \ &{\text {\rm if}}\ v\equiv 3\pmod 4. \end{cases}
\end{equation*}
Moreover we let $\chi$ be a Dirichlet character.
 A {\it harmonic weak Maass form of weight $k$ with Nebentypus $\chi$ on a subgroup
$\Gamma \subset \Gamma_0(4)$} is any smooth function $M:\H\to \C$
satisfying the following:
\begin{enumerate}
\item For all $ \left(\begin{smallmatrix}a&b\\c&d
\end{smallmatrix} \right)\in \Gamma$ and all $\tau \in \H$, we
have
\begin{displaymath}
M \left(\frac{a \tau +b}{c \tau +d} \right)
= \leg{c}{d}\epsilon_d^{-2k} \chi(d)\,(c\tau+d)^{k}\ M(\tau).
\end{displaymath}
\item We  have that $\Delta_k(M)=0$.
\item The function $M$ has
at most linear exponential growth at all the cusps of $\Gamma$.
\end{enumerate}
We let $H_{k}\left(\Gamma, \chi\right)$ denote the space of harmonic weak Maass forms of weight $k$ with Nebentypus $\chi$ on a subgroup $\Gamma$. Every harmonic weak Maass form $M$ uniquely decomposes into a holomorphic and a non-holomorphic part.
 To be more precise, we let
  $\xi_{k}:=2iy^{k}\frac{\overline{\partial}}{\partial\overline{\tau}}$.
This differential operator defines a surjective map
\[
 \xi_{k} : H_{k}\left(\Gamma, \chi\right) \rightarrow M_{2-k}^!\left(\Gamma, \overline{\chi}\right),
\]
where $ M_{2-k}^!\left(\Gamma, \chi\right)$ is the space of weight $2-k$ weakly holomorphic modular forms (i.e., those modular forms that may have poles at the cusps of $\Gamma$)   with Nebentypus $\chi$ on $\Gamma$.
The holomorphic part $M^{+}$ of $M$ is a unilateral Fourier series,
$$
M^{+}(\tau) = \sum_{n \geq n_0} a(n)q^n,
$$
and if   $\xi(M)$ is a cusp form, then  the non-holomorphic part $M^{-}$ is a period integral,
\begin{equation} \label{period}
M^{-}(\tau) =
\int_{-\overline{\tau}}^{i\infty} \frac{g(z)}{(-i(z+\tau))^{k}} dz.
\end{equation}
We then call $g$ the   \emph{shadow} of $M^+$. 
It can be recovered from $M$  by
$\xi_{k}(M)=\xi_{k}(M^{-})=2^{1-k}i \overline{g(-\overline{\tau})} \in S_{2-k}\left(\Gamma, \overline{\chi}\right)$, where $g$ is as in \eqref{period}.

While Ramanujan presented a rough characterization of mock theta functions, all of the examples he wrote down are now known to be holomorphic parts of weight $1/2$ harmonic weak Maass forms.  Following Zagier \cite{Za1}, the holomorphic part of a harmonic weak  Maass form is called a \emph{mock modular form} in general, and a mock theta function when $k=1/2$.


Next we recall important  work of Zwegers \cite{Zw1}.
For $\tau\in \H$, $u, v\in \C \setminus (\mathbb Z\tau + \mathbb
Z)$, Zwegers defined the Lerch sum
\begin{equation*}
\mu(u,v)=
\mu(u,v;\tau):=\frac{z^{\frac12}}{\vartheta(v)} \sum_{n\in
\Z}\frac{(-w)^nq^{n(n+1)/2}}{1-zq^n},
\end{equation*}
where $z:=e^{2\pi i u}$, $w:=e^{2\pi i v}$, $q:=e^{2\pi i \tau}$ and the Jacobi theta function
\begin{equation*}
\vartheta(v;\tau)=\vartheta(v):=\sum_{\nu\in \Z+\frac{1}{2}}e^{\pi i \nu}w^{\nu}
q^{\nu^2/2}.
\end{equation*}
We require the following facts about these functions
\begin{lemma}\label{muProperties}
Assume the notation above.
\begin{enumerate}
\item    We have
\begin{displaymath}
\begin{split}
\vartheta(u+\tau)&=- e^{  - \pi i \tau -2 \pi i u} \vartheta(u), \\
\vartheta(-u)&=- \vartheta(u), \\
\vartheta(u)&= - i q^{\frac18} \, z^{-\frac12}(q)_{\infty} (z)_{\infty} \left(z^{-1}q \right)_{\infty}.
\end{split}
\end{displaymath}
\item We have
\begin{displaymath}
\begin{split}
\mu(u,v)&=\mu(v,u),\\
\mu(u+1,v)&=-\mu(u,v),\\
z^{-1}wq^{-\frac{1}{2}}\mu(u+\tau,v)&=-\mu(u,v)-iz^{-\frac{1}{2}}w^{\frac{1}{2}}q^{-\frac{1}{8}},\\
\mu(u+ \tau, v + \tau)&= \mu(u,v),
 \\
\mu(u+w,v+w)-\mu(u,v)&=
\frac{1}{2 \pi i}
\frac{\vartheta'(0) \, \vartheta(u+v+w)\, \vartheta(w)}{\vartheta(u) \, \vartheta(v)\, \vartheta(u+w)\, \vartheta(v+w)} .
\end{split}
\end{displaymath}
\end{enumerate}
\end{lemma}

Zwegers used $\mu$ to construct harmonic weak Maass forms. To make
this precise, for $\tau\in \H$ and $u\in \C$, let
$c:=\im(u)/y$, and define
\begin{equation*}
R(u)=
R(u;\tau):=\sum_{\nu\in \Z+\frac{1}{2}} (-1)^{\nu-\frac{1}{2}}
\left\{\sgn(\nu)-E\left( (\nu+c)\sqrt{2y}\right) \right\}
e^{-2\pi i \nu u}q^{-\nu^2/2},
\end{equation*}
where $E(x)$ is the odd function
\begin{equation}\label{Efunction}
E(x):=2\int_{0}^{x}e^{-\pi u^2} du=\sgn(x)\left(1-\beta(x^2) \right),
\end{equation}
where for positive real $x$ we let
$\beta(x):=\int_{x}^{\infty}u^{-\frac{1}{2}}e^{-\pi u} du$.

Using
$\mu$ and $R$, Zwegers defined the real analytic function
\begin{equation*}\label{muHat}
\widehat{\mu}(u,v)
=
\widehat{\mu}(u,v;\tau):=\mu(u,v)+\frac{i}{2}R(u-v).
\end{equation*}
This function specializes at torsion points to give weight $1/2$ harmonic weak Maass forms.
This is apparent from the following theorem.

\begin{theorem}\label{ZwegersThm}
Assuming the notation and hypotheses above, we have that for $k,\ell,m, n \in \Z$
\begin{displaymath}
\begin{split}
\widehat{\mu}(u,v)&=\widehat{\mu}(v,u),\\
\widehat{\mu}\left( u+k \tau+ \ell, v + m \tau+n\right)
&= (-1)^{k+\ell+m+n} e^{\pi i(k-m)^2 \tau +2 \pi i (k-m)(u-v)}\,  \widehat{\mu}(u,v).
\end{split}
\end{displaymath}
Moreover, if
$A=\left(\begin{smallmatrix}\alpha&\beta\\\gamma&\delta\end{smallmatrix}\right)\in
\SL_2(\Z)$, then
\begin{displaymath}
\widehat{\mu}\left(\frac{u}{\gamma\tau+\delta},\frac{v}{\gamma\tau+\delta};\frac{\alpha\tau+\beta}
{\gamma\tau+\delta}\right)=
\chi(A)^{-3}(\gamma\tau+\delta)^{\frac{1}{2}}e^{-\pi i
\gamma(u-v)^2/(\gamma\tau+\delta)} \cdot \widehat{\mu}(u,v;\tau),
\end{displaymath}
where
$\chi(A):=\eta\left(\frac{\alpha \tau +\beta}{\gamma \tau +\delta} \right)/\left((\gamma\tau+\delta)^{\frac{1}{2}}\eta(\tau)\right)$.
\end{theorem}

\begin{remark} \label{JacobiRemark}
\emph{Note that from Theorem \ref{ZwegersThm}  one can conclude that for $a,b \in \Q$ the function
\begin{equation} \label{Jacobig}
g(u;\tau):=
z^a\, q^{-\frac{a^2}{2}} \mu(u,a \tau+b)
\end{equation}
can be completed to   a non-holomorphic Jacobi form of weight $\frac12$ and index $-\frac12$ for some subgroup and some
multiplier. To be more precise, it turns out, that one obtains by this completion a so-called harmonic
Maass Jacobi form (see \cite{BR} for the precise definition).}
\end{remark}
The function $R$ can also be written as a theta-integral.
\begin{proposition} \label{IntProp}
For $a,b \in \R$ we define
\[
g_{a,b}(\tau):= \sum\limits_{n\in a+\Z}  n e^{\pi i n^2 \tau+2\pi i n b}.
\]
Then, for $a \in \left(-\frac12,\frac12 \right]$ and $b \in \R$, we have
$$
\int_{- \overline{\tau}}^{i \infty}
\frac{g_{a+\frac12,b+\frac12} (w)}{\sqrt{ -i( \tau+w) }} \, dw
= - e^{- \pi i a^2 \tau + 2 \pi i a \left(b + \frac12 \right) }
R(a \tau-b) + i  \delta_{\frac12,a},
$$
where  $\delta_{\frac12,a}=0$, unless $a=\frac12$ in which case it equals $1$.
\end{proposition}
We note that the case $a \in \left(-\frac12,\frac12 \right)$ is Theorem 1.16 (1)
of \cite{Zw1} and that the case $a= \frac12$ can be proved along the same lines.

We are now ready to prove the main theorem of this section.

\begin{theorem} \label{explicit}
Let $z$ be a root of unity.
\begin{enumerate}
\item
The function
$q^{-\frac13}  N^o(0,0;z;q)$
is  a mock theta function with shadow
\[
 -\frac{i}{\sqrt{6}} \sum\limits_{n \equiv 1 \pmod   3}(-1)^{\frac{n-1}{3}} nq^{\frac{n^2}{3}} (z^{-n}+z^{n})
 .
\]
\item
The function
$q^{-\frac{1}{2}} N^o\left(0,\frac{1}{q};z;q \right)$
is a mock theta function with shadow
\[
- \frac{i}{2} \sum\limits_{n\equiv 1\pmod 2} (-1)^{\frac{n-1}{2}} n q^{\frac{n^2}{2}}z^{-n}
.
\]

\item
The function
$q^{-\frac18} N^o(0,-1;z;q)$
is a mock modular form.
For $z\not=1$ is a mock theta function and  its shadow is given by
\[
\frac{i}{4} \frac{z^{\frac{1}{2}}}{1-z}\sum\limits_{n\equiv 1 \pmod 4} n q^{\frac{n^2}{8}} \left(z^{\frac{n}{2}}+z^{-\frac{n}{2}}\right)
.
\]
For $z=1$ it has weight $\frac32$ and shadow
\[
 \frac{i}{2\pi} \frac{\eta^2(2\tau)}{\eta(\tau)}.
\]

\item
The   function
$N^o(1,-1;z;q)$ is a mock modular form. For  $z\not=\pm1$ it is  a mock theta function.
Its shadow is given by
\[
 -\frac{\sqrt{2}iz}{1-z^2}\sum\limits_{n\in \Z} n q^{n^2} z^{n}
.
\]\
For $z=\pm1$ it is a mock modular form of weight $\frac32$. For $z=1$ the shadow is
\[
-\frac{i}{ 4\sqrt{2}\pi} \Theta(\tau)
\]
and for   $z=-1$  it has the  shadow
\[
\frac{i}{ 4\sqrt{2}\pi} \Theta\left(\tau+\frac12\right).
\]

 \item
 The   function
 $q^{-\frac14} N^o\left( 1,\frac{1}{q};z;q\right)$
 is a mock modular  form.
 For $z\not=-1$ it is a mock theta function and its shadow is given by
 \[
  -\frac{i}{ 2\sqrt{2}}\frac{z^{\frac12} }{1+z}\sum\limits_{n\equiv 1 \pmod 2} (-1)^{\frac{n-1}{2}} n q^{\frac{n^2}{4}} z^{\frac{n}{2}}
.
 \]
For $z=-1$ it is a mock modular form of weight $\frac32$ with shadow
\[
  \frac{-1}{2 \sqrt{ 2}\pi } \frac{\eta^2(4\tau)}{\eta(2\tau)}
.
\]

\end{enumerate}
\end{theorem}


\begin{remark}
\emph{We note that (2)  and (3) of Theorem \ref{explicit} (for $z\not=1$)  could be concluded
from   \cite{Br-On-Rh1}  but for the readers convenience we give a proof here.}
\end{remark}
\begin{proof}
For the proof, we will require the well-known fact  that if $\lambda,\mu \in \Q$ and $\Phi(u;\tau)$ is a Jacobi form of weight $k$ and index $m$, then
$q^{m\lambda^2} \, \Phi(\lambda \tau+ \mu;\tau)$
is a modular form (on some congruence subgroup).
Moreover
$\left. \frac{\partial}{\partial u}\left( \Phi(u;\tau)\right) \right|_{u=0}$
is a modular form of weight $k+1$ and index $m$.

\noindent
(1)
We first consider the case  $(a,b)=(0,0)$.
We have   by (\ref{Watsonid})
$$
N^o(0,0;z;q) = \frac{1}{(q^2;q^2)_{\infty}} \sum_{n \in \mathbb{Z}} \frac{(-1)^nq^{3n^2+3n+1}}{1-zq^{2n+1}}.
$$
 It was shown in
 \cite{BZ}  that
$$
N^o(0,0;z;q)
= z^{-1} \left( R^*(zq;q^2)-1\right),
$$
where $R(z;q)$ is Dyson's 2-variable rank generating function
 and
 $$
 R^*(z;q) := \frac{R(z;q)}{1-z}.
 $$
 Using the identity (see equation (3.1) of \cite{BZ})
 \begin{equation*} \label{Rstar1}
R^*(z;q)=
iz^{-\frac32 }  q^{-\frac18} \mu(3u,-\tau;3 \tau) - i z^{\frac12} q^{-\frac18} \mu(3u,\tau;3 \tau) - iz^{-\frac12} q^{ \frac{1}{24}} \frac{ \eta^3(3 \tau)  }{ \eta(\tau) \vartheta(3u;3 \tau)  }
\end{equation*}
gives that
\begin{multline*}
N^o(0,0;z;q)
=
iz^{-\frac52 }  q^{-\frac74} \mu(3u+3 \tau,-2\tau;6 \tau) - i z^{-\frac12} q^{\frac14} \mu(3u+3 \tau,2\tau;6 \tau)\\
- iz^{-\frac32} q^{ -\frac{5}{12}} \frac{ \eta^3(6 \tau)  }{ \eta(2\tau) \vartheta(3u+3 \tau;6 \tau)  }-z^{-1}.
\end{multline*}
Combining this  with Lemma \ref{muProperties}, Theorem \ref{ZwegersThm}, and Remark \ref{JacobiRemark}  one can show that the
 function
$q^{-\frac13}N^o(0,0;z;q)$
is (up to addition of a constant) the holomorphic part of a   harmonic weak  Maass form
of  weight $\frac12$.  In particular, its   non-holomorphic part  is given by
\[
\frac{1}{2} \left( -q^{-\frac{25}{12}}z^{-\frac{5}{2}}  R(3u+5\tau;6\tau) + q^{-\frac{1}{12}}z^{-\frac{1}{2}} R(3u+\tau;6\tau)\right)
.
\]
Using Proposition \ref{IntProp}, the elliptic transformation properties of $R(u)$ and the properties of $g_{a,b}(\tau)$ given in \cite{Zw1} we obtain that this equals
\begin{align*}
& \frac{1}{2} \left( -2 q^{-\frac{1}{3}}z^{-1}+ q^{-\frac{1}{12}}z^{\frac{1}{2}}R(3u-\tau;6\tau)
 + q^{-\frac{1}{12}}z^{-\frac{1}{2}} R(3u+\tau;6\tau)\right)
\\
& = -q^{-\frac{1}{3}}z^{-1} -\frac{e^{\frac{\pi i}{6} } }{2} \int_{- 6\overline{\tau}}^{i \infty}
\frac{\left(g_{\frac{1}{3},\frac{1}{2}-3u }(w)+g_{\frac{1}{3},\frac{1}{2}+3u }(w)\right)}{\sqrt{ -i(6 \tau+w) }} \, dw
.
\end{align*}
Letting $w\rightarrow 6w$ we see that the shadow of $q^{-\frac13}N^o(0,0;z;q)$ equals
\begin{align*}
 -\frac{\sqrt{6}e^{\frac{\pi i}{6}}}{2}\left(g_{\frac{1}{3},\frac{1}{2}-3u }(6\tau)+g_{\frac{1}{3},\frac{1}{2}+3u }(6\tau)\right)
.
 \end{align*}
 Inserting the definition of $g_{a,b}$ now easily gives the claim.

\noindent
(2)  We next consider the case $(a,b)=\left(0,\frac{1}{q} \right)$.
We have   by  (\ref{Watsonid})
$$
N^o \left( 0,\frac{1}{q};z;q\right)
= \frac{\left(-q^2;q^2 \right)_{\infty}}{\left( q^2;q^2\right)_{\infty}}
\sum_{n \in \Z}
\frac{(-1)^n q^{2n^2+2n+1 } }{1-zq^{ 2n+1} }.
$$
It is not hard to see, proceeding as in the proof of Theorem 3.1 in  \cite{BZ}, that
\begin{equation} \label{case2mu}
N^o \left( 0,\frac{1}{q};z;q\right)
=
- i \frac{\eta^4(4 \tau)}{z \eta^2(2 \tau) \vartheta(2u+2 \tau;4 \tau)}
- i q^{\frac12} \mu(2u + 2 \tau,2 \tau;4 \tau).
\end{equation}
This easily implies the claim similarly as before.
Here, the non-holomorphic part of the completion of  $q^{-\frac12}N^o \left( 0,\frac{1}{q};z;q\right)$ is given by $\frac{1}{2} R(2u;4\tau)$. Using Proposition \ref{IntProp} we obtain
    \[
\frac{1}{2} R(2u;4\tau) = - \frac12 \int_{-4 \overline{\tau}}^{i \infty}
\frac{g_{\frac{1}{2},\frac{1}{2}-2u} (w)}{\sqrt{ -i(4 \tau+w) }} \, dw
.
\]
We let $w\rightarrow 4w$ and see that the shadow of $q^{-\frac12}N^o \left( 0,\frac{1}{q};z;q\right)$ is equal to
$- g_{\frac{1}{2},\frac{1}{2}-2u}(4\tau).$
From this it is not hard to show the claim.

\noindent
(3) For   $(a,b)=(0,-1)$, we  have  from  (\ref{Watsonid})
\begin{equation} \label{case0-1}
N^o \left( 0,-1;z;q\right)
= \frac{\left(q;q^2 \right)_{\infty}}{\left( q^2;q^2\right)_{\infty}}
\sum_{n \in \Z}
\frac{q^{2n^2+3n+1 } }{\left(1-zq^{ 2n+1}\right) \left( 1-q^{2n+1}\right) }.
\end{equation}
We first assume that $z \not=1$.
Using that in this case we have that
$$
\frac{z^{-1}-1}{  \left( 1-zq^{2n+1}\right)\left( 1-q^{2n+1}\right)}
= \frac{z^{-1}}{1-q^{2n+1}}
- \frac{1}{1-zq^{2n+1}}
$$
and
\begin{equation} \label{vanish}
\sum_{n \in \Z}
\frac{q^{2n^2+3n+1 } }{1-q^{ 2n+1}  }=0
\end{equation}
yields
$$
N^o \left( 0,-1;z;q\right)
= -
 \frac{\left(q;q^2 \right)_{\infty}}{\left(z^{-1}-1\right)\left( q^2;q^2\right)_{\infty}}
\sum_{n \in \Z}
\frac{q^{2n^2+3n+1 } }{1-zq^{ 2n+1}}.
$$
As in the proof of Theorem 3.3 of \cite{BZ}, we see that
$$
N^o \left( 0,-1;z;q\right)
= -\frac{1}{1-z}
\left(
\mu\left(2u+2 \tau,\tau+ \frac12;4 \tau\right)
-  z \mu\left(2u+2 \tau,3\tau+ \frac32;4 \tau\right)
\right).
$$
This easily yields the claim for $z \not=1$.
In this case, the non-holomorphic part of the completion of  $q^{-\frac18} N^o \left( 0,-1;z;q\right)$ is
given by
\[
-
 \frac{i}{2}\frac{q^{-\frac{1}{8}}}{1-z} \left(R\left(2u+\tau-\frac{1}{2};4\tau\right)
+zR\left(2u-\tau-\frac{1}{2};4\tau\right)\right)
.
\]
Using Proposition \ref{IntProp} and properties of $g_{a+\frac12,b+\frac12}$ we find that this equals
\begin{align*}
 \frac{i}{2} \frac{z^{\frac12}}{1-z}\int_{- 4\overline{\tau}}^{i \infty}
\frac{ g_{\frac{1}{4},2u}(w) +  g_{\frac{1}{4},-2u}(w)}{\sqrt{ -i(4 \tau+w) }} \, dw
 .
 \end{align*}
Again we let $w\rightarrow 4w$ and we see that the shadow of  $q^{-\frac18} N^o \left( 0,-1;z;q\right)$ is given by
\[
  \frac{iz^{\frac12}}{1-z}\left(g_{\frac{1}{4},2u}(4\tau) -  g_{\frac{1}{4},-2u}(4\tau)\right)
.
\]
Again rewriting gives the claim.


We next consider the case $z=1$.
We have  from (\ref{case0-1})
$$
N^o \left( 0,-1;1;q\right)
=
 \frac{\left(q;q^2 \right)_{\infty}}{\left( q^2;q^2\right)_{\infty}}
\sum_{n \in \Z}
\frac{q^{2n^2+3n+1 } }{\left(1-q^{ 2n+1}\right)^2}.
$$
Following the calculations for $z \not=1$, we see that
\begin{eqnarray*}
N^o \left( 0,-1;1;q\right)
&= &
\left.
\frac{d}{d z}
 \left(
\mu\left(2u+2 \tau,\tau+ \frac12;4 \tau\right)
- z\mu\left(2u+2 \tau,3\tau+ \frac32;4 \tau\right)
\right)\right|_{z=1} \\
&=&\left. q^{\frac18} \frac{d}{d z} \left( \Phi_1(u;\tau)\right)\right|_{z=1}
+ \frac12 q^{-\frac18}\Phi_1(0;\tau),
\end{eqnarray*}
where
$$
 \Phi_1(u;\tau)
 :=
z^{-\frac12}q^{-\frac18}
  \left(
\mu\left(2u+2 \tau,\tau+ \frac12;4 \tau\right)
- z\mu\left(2u+2 \tau,3\tau+ \frac32;4 \tau\right)
\right).
$$
Using (\ref{vanish}), we obtain that
\begin{eqnarray*}
q^{-\frac18}
N^o \left( 0,-1;1;q\right)
=   \left.
\frac{d}{d z} \left(\Phi_1(u;\tau)\right)\right|_{z=1}.
\end{eqnarray*}
Using the above, one can show that $q^{-\frac18}
N^o \left( 0,-1;1;q\right) $ can be completed to  a harmonic weak  Maass form by adding the term
\begin{multline} \label{diff1}
 \left.
\frac{i}{2} \frac{d}{d z}
\left(
q^{-\frac18}z^{-\frac12} R\left( 2u+ \tau-\frac12;4 \tau\right)
+q^{-\frac18}z^{\frac12} R\left( 2u- \tau-\frac12;4 \tau\right)
\right)\right|_{z=1} \\
=
\left.
i \frac{d}{d z}
\left(
q^{-\frac18}z^{-\frac12} R\left( 2u+ \tau-\frac12;4 \tau\right)
\right)
\right|_{z=1}
\\
=
 - \left.
\frac{d}{d z} \left(
\sum_{n \in \Z}
\left( \text{sgn}\left(n-\frac12 \right) -E \left(\left( n-\frac14+\frac{\text{Im}(u)}{2y}\right) \sqrt{8y} \right) \right)
q^{-2 \left(n-\frac14 \right)^2} z^{-2 \left(n-\frac14 \right) } \right)
\right|_{z=1}
.
\end{multline}
Using the   identities (\ref{Efunction}) and
\begin{displaymath}
\begin{split}
\beta(x)&=\frac{1}{\pi}x^{-\frac{1}{2}}e^{-\pi
x}-\frac{1}{2\sqrt{\pi}}\cdot \Gamma\left(-\frac12;\pi x\right),\\
E'(x)&=2e^{-\pi x^2},
\end{split}
\end{displaymath}
where $\Gamma(\alpha;x):=\int_{x}^{\infty}t^{\alpha-1}e^{-t} dt$ is the usual incomplete gamma-function, we compute that
(\ref{diff1}) equals
\begin{equation} \label{nonhol3}
-\frac{1}{\sqrt{\pi}} \sum_{n \in \Z} \left|n-\frac14 \right|
\Gamma\left( -\frac12; 8 \pi y \left( n-\frac14\right)^2\right)q^{-2 \left(n-\frac14 \right)^2}
\end{equation}
which does not contribute to the holomorphic part.
To rewrite (\ref{nonhol3})   as a theta integral, we use the easily verified identity ($\alpha>0$)
\begin{equation} \label{nonholintegral}
e^{- \alpha i \tau} \Gamma\left( -\frac12;2 \alpha y\right)
= -\frac{i}{\sqrt{\alpha}} \int_{-\overline{\tau}}^{i \infty}
\frac{e^{\alpha i t}}{\left(-i \left(t + \tau \right) \right)^{\frac32}}dt.
\end{equation}
This yields that (\ref{nonhol3}) may be written as
\[
 \frac{i}{2\pi}\int_{-\overline{\tau}}^{i\infty} \frac{\sum_{n \in \Z}
 e^{4\pi i \left(n-\frac14\right)^2 t}}{(-i(t+\tau))^{\frac32}} dt.
\]
This gives that the  shadow of $q^{-\frac18}
N^o \left( 0,-1;1;q\right) $ equals
\[
\frac{i}{2\pi}\sum_{n \in \Z}
 q^{\frac{\left(4n-1\right)^2}{8}}=
\frac{i}{2\pi} \frac{\eta^2(2\tau)}{\eta(\tau)}
.
\]

 \noindent
(4) Turing to the   case  $(a,b)=(1,-1)$, we have from  (\ref{Watsonid})
$$
N^o(1,-1;z;q) = \frac{(q^2;q^4)_{\infty}}{(q^4;q^4)_{\infty}}
 \sum_{n \in \mathbb{Z}} \frac{q^{n^2+3n+1}}{\left(1-zq^{2n+1}\right)\left( 1-q^{4n+2}\right)}.
$$
We first assume that $z \neq \pm 1$.
Then we have
$$
 \frac{1}{\left(1-zq^{2n+1}\right)\left( 1-q^{4n+2}\right)}
 = \frac{-z^2}{\left(1-z^2 \right)\left( 1-zq^{2n+1}\right)}
 +  \frac{1}{2\left(1-z \right)\left( 1-q^{2n+1}\right)}
 +  \frac{1}{2\left(1+z \right)\left( 1+q^{2n+1}\right)} .
$$
Since
\begin{eqnarray} \label{thetasum}
 \sum_{n \in \mathbb{Z}} \frac{q^{n^2+3n+1}}{ 1-q^{2n+1}}
& =& - \frac12 \sum_{n \in \Z}q^{n^2+n}
 = -  \frac{(q^4;q^4)_{\infty}^2}{(q^2;q^2)_{\infty}}, \\ \nonumber
  \sum_{n \in \mathbb{Z}} \frac{q^{n^2+3n+1}}{ 1+q^{2n+1}}
 &=&  \frac{(q^4;q^4)_{\infty}^2}{(q^2;q^2)_{\infty}},
\end{eqnarray}
we obtain
$$
  N^o(1,-1;z;q) =
 -\frac{z}{1-z^2}
 -
  \frac{(q^2;q^4)_{\infty}z^2}{(1-z^2)(q^4;q^4)_{\infty}} \sum_{n \in \mathbb{Z}} \frac{q^{n^2+3n+1}}{1-zq^{2n+1}}.
$$
A direct computation shows  that
\begin{equation} \label{equation1-1}
N^o(1,-1;z;q)
=-  \frac{z}{1-z^2}
+2 z^{\frac32} q^{-\frac14}
\frac{1}{1-z^2}
\mu\left(u-\tau, \frac12;2 \tau \right).
\end{equation}
From this it is not hard to conclude  that for $z\not= \pm1$ the function $N^o(1,-1;z;q)$  can be completed to a harmonic weak Maass form of weight $\frac12$, by adding   the function
\begin{equation} \label{adding}
-
\frac{z}{1-z^2}
+
z^{\frac12} q^{-\frac14}
\frac{i}{1-z^2}
R\left(u+\tau-\frac12; 2 \tau\right).
\end{equation}
Using  Proposition \ref{IntProp} we find
\begin{align*}
z^{\frac12} q^{-\frac14}
\frac{i}{1-z^2}  R\left(u+\tau-\frac{1}{2};2\tau\right)
= \frac{z}{1-z^2} - i \frac{z}{1-z^2} \int_{- 2\overline{\tau}}^{i \infty}
\frac{g_{0,u} (w)}{\sqrt{ -i(2 \tau+w) }} \, dw
.
\end{align*}
Thus, the holomorphic contribution to (\ref{adding}) is $0$ and the shadow of $N^o(1,-1;z;q)$
equals
$$
\frac{\sqrt{2}i z}{1-z^2} g_{0,u}(2\tau)
.
$$

We next turn to the case $z=1$ (the case $z=-1$ is obtained by replacing $q$ by $-q$ and then multiplying by $-1$).
We have by taking the limit of (\ref{equation1-1})
 \begin{equation*}
N^o(1,-1;1;q)
=
\frac12
\left.
- \frac{d}{d z} \left(  z^{\frac32} q^{-\frac14}
\mu\left(u-\tau, \frac12;2 \tau \right)\right)\right|_{z=1}
=
-
\left.
 \frac{d}{d z} \left(   \Phi_2(u;\tau)\right)\right|_{z=1}
,
\end{equation*}
where
$$
\Phi_2(u;\tau):=  z^{\frac12} q^{-\frac14}
\mu\left(u-\tau, \frac12;2 \tau \right).
$$
Here we used that $\Phi_2(0;\tau)=\frac12$.
From the case $z \not=1$ we may conclude that this function can be completed to an harmonic weak Maass form by adding
\begin{multline} \label{correction}
- \frac{i}{2} \left.
\frac{d}{d z}
\left(
z^{\frac12} q^{-\frac14}
R\left(u-\tau-\frac12; 2 \tau\right)
\right) \right|_{z=1} \\
=
 \frac12 \left. \frac{d}{d z} \left(
 \sum_{n \in \Z} \left(\text{sgn}\left( n+\frac12\right)
 -E \left( \left( n+\frac{\text{Im}}{2y}\right) 2 \sqrt{y} \right)
 q^{-n^2}z^{-n}
  \right) \right)
 \right|_{z=1}\\
 =
  \frac{1}{ 4\sqrt{\pi}}
  \sum_{n \in \Z}
  \Gamma\left( -\frac12;4 \pi n^2y\right) |n| q^{-n^2}.
\end{multline}
We now proceed as in case (3). Using (\ref{nonholintegral}), we  rewrite the correction term as a period integral
\[
 -\frac{i}{ 4\sqrt{2}\pi} \int_{-\overline{\tau}}^{i\infty}  \frac{\sum_{n \in \Z} e^{2\pi i n^2 t}}{(-i(t+\tau))^{\frac32}} dt
.
\]
From this we can directly conclude that the shadow of $N^o(1,-1;1;q)$ is given by
\[
-\frac{i}{ 4\sqrt{2}\pi} \Theta(\tau).
\]

\noindent
(5)
Finally we treat the case  $(a,b)=(1,1/q)$.
We obtain  that
 \begin{equation*} \label{1-1/q}
 N^o(1,1/q;z;q) =
\frac{(-q)_{\infty}}{(q)_{\infty}} \sum_{n \in \mathbb{Z}} \frac{(-1)^nq^{n^2+2n+1}}{\left(1-zq^{2n+1}\right)\left(1+q^{2n+1} \right)} .
\end{equation*}
We first assume that   $z \neq -1$.
Using the identity
$$
 \frac{1+z^{-1}}{\left(1-zq^{2n+1}\right)\left( 1+q^{2n+1}\right)}
  =   \frac{z^{-1}}{\left( 1+q^{2n+1}\right)}
 +  \frac{1}{ 1-zq^{2n+1}}
$$
and the fact that
\begin{equation} \label{sumvanish}
 \sum_{n \in \mathbb{Z}} \frac{(-1)^nq^{n^2+2n+1}}{1+q^{2n+1} }
 = 0
 \end{equation}
 gives that
\begin{equation*} \label{N11q}
N^o(1,1/q;z;q) =
 \frac{(-q)_{\infty}}{(1+1/z)(q)_{\infty}} \sum_{n \in \mathbb{Z}} \frac{(-1)^nq^{n^2+2n+1}}{1-zq^{2n+1}} .
\end{equation*}
It is not hard to see that
this  can be rewritten as
$$
N^o(1,1/q;z;q) =
-\frac{iz^{\frac12}}{1+z}  q^{\frac14}
\mu(u + \tau,\tau;2 \tau).
$$
It is now not hard to show that $q^{-\frac14}N^o(1,1/q;z;q) $ is the holomorphic part of a harmonic weak  Maass form.
Here, the associated non-holomorphic part is $\frac{1}{2(1+z)} z^{\frac{1}{2}} R(u;2\tau)$. By Proposition \ref{IntProp} we find
\[
\frac{z^{\frac{1}{2}}}{2(1+z)} R(u;2\tau) = - \frac{z^{\frac{1}{2}}}{2(1+z)}   \int_{- 2\overline{\tau}}^{i \infty}
\frac{g_{\frac12,\frac12-u} (w)}{\sqrt{ -i(2 \tau+w) }} \, dw
,
\]
which obviously does not have a holomorphic contribution.
Hence, the shadow of $q^{-\frac14}N^o(1,1/q;z;q) $ is given by the unary theta function
\[
-\frac{1}{ \sqrt{2}}\frac{z^{\frac12} }{1+z} g_{\frac{1}{2}, \frac{1}{2}-u}(2\tau)
.
\]

We next deal with  the case $z=-1$
From the case $z \not=-1$, we may conclude that
\begin{equation*}
N^o(1,1/q;-1;q) =
\frac{d}{d z}
\left(
\left.
-iz^{\frac12}q^{\frac14}
\mu(u + \tau,\tau;2 \tau) \right)
\right|_{z=-1}.
\end{equation*}
Thus, using (\ref{sumvanish}), we obtain
$$
q^{-\frac14} N^o(1,1/q;-1;q) =
 -i \left.
\frac{d}{d z}  \left(\Phi_3(u;\tau)\right)\right|_{z=1}
- \frac{i}{2} \mu\left( \frac12+ \tau,\tau;2 \tau\right)
=  -i \left.
\frac{d}{d z} \left( \Phi_3(u;\tau) \right)\right|_{z=1},
$$
where
$$
\Phi_3(u;\tau):= \mu \left( u+\frac12+\tau, \tau; 2 \tau \right).
$$
Thus $q^{-\frac14} N^o(1,1/q;-1;q)$ may
 be completed to a harmonic weak Maass form  by adding the term
\begin{eqnarray*}
\left.
\frac{1}{2} \frac{d}{d z} \left(
R\left( u+\frac12;2 \tau\right)\right) \right|_{z=1}
&=  &
- \left.
\frac{i}{2} \frac{d}{d z} \left(
\sum_{\nu \in \frac12+\Z}
\left(
\text{sgn}(\nu)- E \left(\left(\nu + \frac{\text{Im}(u)}{2y} \right) 2 \sqrt{y} \right)
\right) q^{-\nu^2}z^{-\nu} \right)\right|_{z=1}\\
&=& \frac{-i}{4 \sqrt{ \pi}}
\sum_{\nu \in \frac12 + \Z} |\nu| \Gamma\left( -\frac12;4 \pi \nu^2y\right) q^{-\nu^2}
\end{eqnarray*}
which does not have a holomorphic contribution.
Using (\ref{nonholintegral}), we  rewrite the non-holomorphic part in terms of a period integral
\[
-  \frac{1}{4 \sqrt{ 2}\pi }
\int_{-\overline{\tau}}^{i\infty}\frac{ \sum_{n \in \frac12 + \Z} e^{2\pi i n^2t}} {(-i(t+\tau))^{\frac32}} dt
\]
and we directly see that the shadow of $q^{-\frac14} N^o(1,1/q;-1;q)$ is given by
\[
- \frac{1}{4 \sqrt{ 2}\pi }
 \sum_{n \in \Z} q^{\frac{\left(2n+1\right)^2}{4}}
 =
- \frac{1}{2 \sqrt{ 2}\pi }\frac{\eta^2(4\tau)}{\eta(2\tau)}
.
\]

\end{proof}

\section{Relation to class numbers}
In this section we prove Theorem \ref{intro1.5}.  Equations \eqref{classeq1} and \eqref{classeq2} follow directly from \eqref{Watsonid} and identities of Andrews, Humbert, and Kronecker.  Equation \eqref{classeq3} we could not trace directly to the literature and so we argue using \eqref{Watsonid}, an identity of Watson, and Lemma \ref{muProperties}.
\begin{proof}[Proof of Theorem \ref{intro1.5}]
For \eqref{classeq1}, beginning with \eqref{Watsonid} we have
\begin{eqnarray*}
N^o(1,-1;1;q) &=& \frac{(q^2;q^4)_{\infty}}{2(q^4;q^4)_{\infty}} \sum_{n \in \mathbb{Z}} \frac{q^{n^2+3n+1}}{(1-q^{2n+1})^2} \\
&=& \frac{(q^2;q^4)_{\infty}}{(q^4;q^4)_{\infty}} \sum_{n \geq 0} \frac{q^{n^2+3n+1}}{(1-q^{2n+1})^2} \\
&=& \sum_{n \geq 1} 2F(n)q^n,
\end{eqnarray*}
the last equality being equation (XI) in \cite{Kr1}.

For \eqref{classeq2} we begin with \eqref{Watsonid}, obtaining
\begin{eqnarray*}
N^o(0,-1;1;q) &=& \frac{(q;q^2)_{\infty}}{(q^2;q^2)_{\infty}} \sum_{n \in \mathbb{Z}}\frac{q^{2n^2+3n+1}}{(1-q^{2n+1})^2} \\
&=& \frac{(-q)_{\infty}^2}{(q)_{\infty}^2} \sum_{n \geq 1} \frac{(-1)^{n+1}n^2q^{n(n+1)/2}}{(1+q^n)} \\
&=& \sum_{n \geq 1} F(8n-1)q^n,
\end{eqnarray*}
the penultimate equality being an identity from Ramanujan's lost notebook proven by Andrews \cite[Eq. (1.1)]{An.7} and the final equality coming from Humbert \cite[p. 368]{Hu1} (or see \cite[p. 51]{Wa2}).  To finish we note that $H(8n-1) = F(8n-1)$.

For \eqref{classeq3}, we start by noting that from \eqref{Watsonid} we have
\begin{equation} \label{class1}
N^o(1,1/q;-1;q) = \frac{(-q)_{\infty}}{(q)_{\infty}} \sum_{n \in \mathbb{Z}} \frac{(-1)^nq^{n^2+2n+1}}{(1+q^{2n+1})^2}.
\end{equation}
Next, the first identity in part ($2$) of Lemma \ref{muProperties} is equivalent to the identity
$$
\frac{1}{(z,q/z)_{\infty}}\sum_{n \in \mathbb{Z}} \frac{(-1)^nz^nq^{n(n+1)/2}}{(1-aq^n)} = \frac{1}{(a,q/a)_{\infty}}\sum_{n \in \mathbb{Z}} \frac{(-1)^na^nq^{n(n+1)/2}}{(1-zq^n)}.
$$
In this identity let $q \to q^2$, $a \to 1/q$, and $z\to -zq$ and then apply $\frac{d}{dz}|_{z=1}$ to both sides to obtain
\begin{equation} \label{class2}
\frac{(-q)_{\infty}}{(q)_{\infty}} \sum_{n \in \mathbb{Z}} \frac{(-1)^nq^{n^2+2n+1}}{(1+q^{2n+1})^2} = \frac{1}{(-q,-q,q^2;q^2)_{\infty}} \sum_{n \in \mathbb{Z}} \frac{nq^{n^2+2n-1}}{(1-q^{2n-1})}.
\end{equation}
Finally an identity of Watson \cite[Eq. (3.04), corrected]{Wa2} implies that
\begin{eqnarray*}
\sum_{n \geq 1} F(4n-1)(-q)^n &=& \frac{-1}{(-q,-q,q^2;q^2)_{\infty}} \sum_{n \in \mathbb{Z}} \frac{(n-1/2)q^{n^2}}{(1-q^{2n-1})} \\
&=& \frac{-1}{(-q,-q,q^2;q^2)_{\infty}} \sum_{n \in \mathbb{Z}} \frac{nq^{n^2}}{(1-q^{2n-1})} \\
&=& \frac{-1}{(-q,-q,q^2;q^2)_{\infty}} \sum_{n \in \mathbb{Z}} \frac{nq^{n^2+2n-1}}{(1-q^{2n-1})}.
\end{eqnarray*}
Now apply \eqref{class1} and \eqref{class2} to obtain the first part of \eqref{classeq3}.  To finish we use the fact that $F(8n+3) = 3H(8n+3)$ and $F(8n+7) = H(8n+7)$.
\end{proof}


\section{Quasimock modular forms}
To prove Theorem \ref{intro2}, we will appeal to certain partial differential equations arising from the application of the heat operator to non-holomorphic Jacobi forms \cite{BZ}.
\\
(1) We begin with the case $(a,b)=(0,0)$.  In equation (4.2) of \cite{BZ} it is shown that
$$
\left(6\pi i \frac{\partial}{\partial \tau} + \frac{\partial^2}{\partial u^2}\right)
\left(q^{-\frac{1}{3}}N^o(0,0;z;q)\right)
= - \frac{8q^{-\frac{3}{4}}z^{-\frac{3}{2}}\pi^2i\eta^8(2\tau)}{\vartheta^3(u+\tau;2\tau)}.
$$
We operate on both sides by $ \frac{\partial^{2\ell}}{\partial u^{2\ell}} | _{u=0}$  to obtain
$$
\frac{\partial^{2(\ell+1)}}{\partial u^{2(\ell+1)}} \left(q^{-\frac{1}{3}}N^o(0,0;z;q)\right)  \bigg | _{u=0}=
-6\pi i \frac{\partial}{\partial \tau}\left(\frac{\partial^{2\ell}}{\partial u^{2\ell}}  \left(q^{-\frac{1}{3}}N^o(0,0;z;q)\right)\right) \bigg | _{u=0}+  g_{\ell}(\tau),
$$
where as in the proof of Colloary \ref{introcor2} in Section 3 we have that $g_{\ell}$ is a quasimodular form.  Now by induction (the fact that $q^{-\frac{1}{3}}N^o(0,0;0;q)$ is a mock theta function-as was shown in Theorem 4.6 (1)- settling the case $\ell = 0$), the claim follows.  The rest of the cases are similar and so we will only exhibit the required PDE's.

\noindent
(2) We next consider the case $(a,b)=(0,1/q)$.  Here we make a change of variables in formula (3.6) in \cite{BZ} and can compute, using
(\ref{case2mu}),
\begin{align*}
&\left(4\pi i \frac{\partial}{\partial \tau}+ \frac{\partial^2}{\partial u^2}\right) \left( q^{-\frac{1}{2}} N^{o}\left(0,1/q;z;q\right)\right)
\\
&\quad = \left. - i q^{-\frac{1}{2}}z^{-1} \left(8\pi i \frac{\partial}{\partial \tau}+ \frac{\partial^2}{\partial u^2}\right)\left(\frac{\eta^4(2\tau)}{\eta^2(\tau)\vartheta(2u;2\tau)}+q^{-\frac{1}{4}}z\mu(2u,\tau;2\tau)\right)\right|_{\tau\rightarrow 2\tau, u\rightarrow u+\tau}
\\
&\quad=  4 i  \pi^2  q^{-\frac{1}{2}}z^{-1} \frac{\eta^8(2\tau)}{\eta(4\tau)} \frac{\vartheta\left(u+\tau+1/2;2\tau\right)}{ \vartheta^3(u+\tau;2\tau)}
\text{.}
\end{align*}
Then the claim follows as above.

\noindent
(3) Turning to  the case $(a,b)=(0,-1)$, a  change of variables in formula (3.10) of \cite{BZ} yields the following PDE
\begin{align*}
&\left(4\pi i\frac{\partial}{\partial\tau} +\frac{\partial^2}{\partial u^2}\right)
\left(
e^{\frac{3\pi i}{8}} (1-z)z^{-\frac{1}{2}}q^{-\frac{1}{8}}N^{o}(0,-1;z;q)
\right)
\\
&\quad=\left. z^{-1}q^{-\frac{1}{2}}\left(4\pi i \frac{\partial}{\partial \tau} +\frac{\partial^2}{\partial u^2}\right) \left( z^{\frac{1}{2}} q^{-\frac{1}{8}} \mu(2u,\tau;4\tau)+z^{\frac{3}{2}} q^{-\frac{9}{8}} \mu(2u,3\tau;4\tau)\right)\right|_{u\rightarrow u+\tau, \tau\rightarrow\tau+1/2}
\\
&\quad=-8\pi^2 e^{\frac{3 \pi i}{8}} z^{-\frac{3}{2}}q^{-\frac{3}{4}} \frac{\eta^8(2\tau)}{\eta(\tau)} \frac{\vartheta(u;2\tau)}{  \vartheta^3(u+\tau;2\tau)}
\text{.}
\end{align*}
As in the proof of Theorem 4.2 of \cite{Br-Lo-Os1} we can conclude a PDE for $q^{-\frac{1}{8}}N^{o}(0,-1;z;q)$.

\noindent
(4) For  the case $(a,b)=(1,-1)$, we  use Theorem 1.2 of \cite{BZ} with $\alpha=0$ and $\beta=\frac{1}{2}$.
Making a change of variables   and computing the resulting functions $\vartheta_0, \vartheta_1, a_0, a_1$ occurring there  yields
\begin{align*}
 &\left(2\pi i \frac{\partial}{\partial \tau} + \frac{\partial^2}{\partial u^2} \right)
 \left(\frac{1}{2}(1-z^2)z^{-1} \left(N^{o}(1,-1;z;q)+\frac{z}{1-z^2}\right)
 \right)\\
 &
 = \left. z^{\frac{1}{2}}q^{-\frac{1}{4}} \left(4\pi i \frac{\partial}{\partial \tau}+  \frac{\partial^2}{\partial u^2}\right) \left( \mu\left(u,1/2;\tau\right)\right) \right|_{u\rightarrow u-\tau, \tau \rightarrow 2\tau}
 \\ &
=-16\pi^2z^{\frac{3}{2}}q^{-\frac{3}{4}} \frac{ \eta^6(2\tau) \eta^3(4\tau) }{\vartheta^3(u-\tau;2\tau)\vartheta^2(1/2;2\tau)} \vartheta(2u;4\tau)\text{.}
\end{align*}
Then we argue as in the preceding case.

\noindent
(5) Finally, we consider the case $(a,b)=(1,1/q)$.  By  making a change of variables in formula (3.8) of \cite{BZ} we have
\begin{align*}
&\left(2\pi i \frac{\partial}{\partial \tau} +  \frac{\partial^2}{\partial u^2} \right)
\left( i(1+z)q^{-\frac{1}{4}} z^{-\frac{1}{2}} N^{o}(1,1/q;z;q)
\right)\\
&
=\left.z^{-\frac{1}{2}}q^{-\frac{1}{4}} \left(2\pi i \frac{\partial}{\partial \tau}+  \frac{\partial^2}{\partial u^2} \right)\left(z^{\frac{1}{2}} q^{-\frac{1}{4}} \mu(u,\tau;2\tau)\right)\right|_{u\rightarrow u+\tau}
\\
&= -4\pi^2z^{-\frac{3}{2}}q^{-\frac{3}{4}} \frac{\eta^8(2\tau)}{\eta(\tau)}\frac{\vartheta(u+1/2;\tau)}{\vartheta^3(u+\tau;2\tau)}
\text{.}
\end{align*}
The claim concludes as before.

\section{conclusion} \label{ConclusionSection}
There are many possible number-theoretic applications of the automorphic structure of \newline $N^o(a,b;z;q)$ and $\mathcal{N}^o_{2k}(a,b;q)$.  To give a few examples, the quasimock modularity of $\mathcal{N}^o(a,b;q)$ and the associated PDE's may be used to deduce asymptotic expansions, congruences, and rank moment identities as in \cite{At-Ga1,Br.5,Br1,Br-Ga-Ma1,Br-Lo-Os1,BZ}; the connection to class numbers in Theorem \ref{intro1.5} yields congruences, asymptotics, exact formulas and identities as in \cite{Br-Lo2}; congruences, asymptotics and identities for the mock theta functions in Theorem \ref{intro1} can be deduced as \cite{Br.7,Br1, BK, Br-Lo1,Br-Lo-Os1,Br-On-Rh1,Br-On.5,Br-On1,Ga1,Wal1}; and rank differences and congruences for $N^o(a,1/a;1;q)$ may be studied as in \cite{Br-Lo1.5}.
Also, given the families of mock theta functions studied here and in \cite{Br-Lo1,Br-On1,Br-On-Rh1}, it should be possible to produce many $q$-series identities by canceling the non-holomorphic parts of the corresponding harmonic weak  Maass forms and computing the resulting weakly holomorphic modular form.

There are also interesting combinatorial questions arising from our work.  To give one example, is it possible to make a careful study of the Durfee symbols counted by $N^o(0;-1;1;q)$ and prove combinatorially that $H(8n-1) = 1$ if and only if $n=1$?


\begin{thebibliography}{99}
\bibitem{An.5}
G. Andrews, Mordell integrals and Ramanujan's ``lost" notebook, Lecture Notes in Math. {\bf 899}, Springer, Berlin, 1981,  10--48.
\bibitem{An.7}
G. Andrews, Bailey chains and generalized Lambert series: I. Four identities of Ramanujan,
\emph{Illinois J. Math.} {\bf 36} (1992), 251--274.
\bibitem{An1}
G. Andrews, Partitions, Durfee symbols, and the Atkin-Garvan
moments of ranks, \emph{Invent. Math.} {\bf 169} (2007), 37--73.
\bibitem{At-Ga1}
A.  Atkin and F.  Garvan, Relations between the ranks and cranks of partitions, \emph{Ramanujan J.}  {\bf 7}  (2003), 343--366.
\bibitem{Br.5}
K. Bringmann, On the explicit construction of higher deformations of partition statistics,
\emph{Duke Math. J.} {\bf 144} (2008), 195-233.
\bibitem{Br.7}
K. Bringmann, Asymptotics for rank partition functions,
\emph{Trans. Amer. Math. Soc.} {\bf 361}  (2009), 3483-3500.
\bibitem{Br1}
K. Bringmann, On certain congruences for Dyson's ranks, \emph{Int. J. Number Theory} {\bf 5} (2009), 573- 584.
\bibitem{Br-Ga-Ma1}
K. Bringmann, F. Garvan, and K. Mahlburg, Paritition statistics and quasiharmonic Maass forms, \emph{Int. Math. Res. Not.} (2008), rmn124.
\bibitem{BK} K. Bringmann and B. Kane, Inequalities for differences of Dyson´s rank for all odd moduli, \emph{
Math. Research Letters}, accepted  for publication.
\bibitem{Br-Lo1}
K. Bringmann and J. Lovejoy, Dyson's rank, overpartitions, and weak Maass forms, \emph{Int. Math. Res. Not.} (2007), rnm063.
\bibitem{Br-Lo1.5}
K. Bringmann and J. Lovejoy, Rank and congruence for overpartition pairs,
\emph{Int. J. Number Theory}  {\bf 4} (2008), 303-322.
\bibitem{Br-Lo2}
K. Bringmann and J. Lovejoy, Overpartitions and class numbers of binary quadratic forms, \emph{Proc. Nat. Acad. Sci. USA} {\bf 106} (2009), 5513-5516.
\bibitem{Br-Lo-Os1}
K. Bringmann, J. Lovejoy, and R. Osburn, Automorphic properties of generating functions for generalized rank moments and Durfee symbols, \emph{Int. Math. Res. Not.} (2010), rnp131.
\bibitem{Br-Lo-Os2}
K. Bringmann, J. Lovejoy, and R. Osburn, Rank and crank moments for overpartitions, \emph{J. Number Theory} {\bf 129} (2009), 2567-2574.
\bibitem{Br-On.5}
K. Bringmann and K. Ono, The $f(q)$ mock theta function conjecture and partition ranks,
\emph{Invent. Math.} {\bf 165} (2006), 243-266.
\bibitem{Br-On1}
K. Bringmann and K. Ono, Dyson's rank and weak Maass forms,
\emph{Ann. Math.} {\bf 171} (2010), 419-449.
\bibitem{Br-On-Rh1}
K. Bringmann, K. Ono, and R. Rhoades, Eulerian series as modular forms,
\emph{J. Amer. Math. Soc.} {\bf 21} (2008), 1085--1104.
\bibitem{BR} K. Bringmann and O. Richter, Zagier-type dualities and lifting maps for harmonic Maass-Jacobi forms,  Advances of Math, accepted for publication.
\bibitem{BZ} K. Bringmann and S. Zwegers,
Rank-crank type PDE's and non-holomorphic Jacobi forms,
\textit{Math. Res. Lett.} {\bf 17} (2010),  589-600.

\bibitem{BF} J.  Bruinier and J. Funke,
On two geometric theta lifts, \emph{Duke Math. J.} \textbf{125}
(2004),  45-90.

\bibitem{Ei-Za1}
M. Eichler and D. Zagier, Jacobi Forms, Progress in Mathematics 55, Birkh\"auser, Boston, 1985.
\bibitem{Ga1}
S. Garthwaite, Coefficients of the $\omega(q)$ mock theta function, \emph{Int. J. Number Theory} {\bf 4} (2008), 1027-1042.
\bibitem{Ga-Ra1}
G. Gasper and M. Rahman, Basic hypergeometric series, Cambridge Univ. Press, Cambridge, 1990.
\bibitem{Go-Mc1}
B. Gordon and R. McIntosh, Some eighth order mock theta functions, \emph{J. London Math. Soc.} {\bf 62} (2000), 321--335.
\bibitem{Hi1}
K. Hikami, Transformation of the ``second" order mock theta function, \emph{Lett. Math. Phys.} {\bf 75} (2006), 93--98.
\bibitem{Hu1}
G. Humbert, Formules relatives aux nombres de classes des formes quadratiques binaires et positives,
\emph{J. Math. Pures Appl. (6)} {\bf 3} (1907), 337--449.
\bibitem{Kr1}
M. Kronecker, \"Uber die Anzahl der verschiedenen Classen
quadratischer Formen von negativer Determinante, \emph{J. reine
Angew. Math.} {\bf 57} (1860), 248--255.
\bibitem{Mc1}
R. McIntosh, Second order mock theta functions, \emph{Canad. Math. Bull.} {\bf 50} (2007), 284--290.
\bibitem{Ono}
K. Ono, The web of modularity: arithmetic of the coefficients of modular forms and q-series, CBMS
Regional Conference Series in Mathematics, \textbf{102}, Amer. Math. Soc., Providence, RI, 2004.
\bibitem{Ra1}
S. Ramanujan, The Lost Notebook and Other Unpublished Papers, Narosa Publishing House, New Delhi, 1988.
\bibitem{Wal1}
M. Waldherr, On certain explicit congruences for mock theta functions, preprint.
\bibitem{Wa1}
G. Watson, The final problem: An account of the mock theta functions, \emph{J. London Math. Soc.} {\bf 11} (1936), 55-80.
\bibitem{Wa2}
G. Watson, Generating functions of class-numbers, \emph{Compositio Math.} {\bf 1} (1935), 39-68.
\bibitem{Za1}
D. Zagier, Ramanujan's mock theta functions and their applications d'apr\`es Zwegers and Bringmann-Ono, \emph{S\'eminaire Bourbaki $60$\`eme ann\'ee} (986) (2006-2007), http://www.bourbaki.ens.fr/TEXTES/986.pdf.
\bibitem{Zw} S.  Zwegers, \emph{Mock $\vartheta$-functions
and real analytic modular forms}, $q$-series with applications to
combinatorics, number theory, and physics (Ed. B. C. Berndt and K.
Ono), Contemp. Math. \textbf{291}, Amer. Math. Soc., (2001), pages
269-277.
\bibitem{Zw1}
S. Zwegers, Mock Theta Functions, PhD thesis, Utrecht, 2002.
\end{thebibliography}
\end{document}